\documentclass[]{amsart}
\usepackage{amsmath,amsthm,hyperref,amsfonts,amssymb}
\usepackage{pxfonts}
\hypersetup{
    colorlinks=true,       
    linkcolor=blue,          
    citecolor=cyan,        
}
\usepackage{verbatim}
\usepackage{url}

\makeatother

\newtheorem{theorem}{Theorem}[section]
\newtheorem{lemma}{Lemma}[section]
\newtheorem{corollary}[lemma]{Corollary}
\newtheorem{proposition}[lemma]{Proposition}

\theoremstyle{remark}

\usepackage{tikz}
\usetikzlibrary{decorations.markings}
\makeatother

\def\P{\partial}

\def\ri{\mathrm i}

\def\rb{\mathbb R}

\def\zb{\mathbb Z}

\numberwithin{equation}{section}

\title{Monotonicity properties related to the ratio of two gamma functions}
\author{Nian Hong Zhou}
\address{School of Mathematics and Statistics, Guangxi Normal University\\
No.1 Yanzhong Road, Yanshan District, Guilin, 541006\\
Guangxi, PR China}
\email{nianhongzhou@outlook.com; nianhongzhou@gxnu.edu.cn}

\author{Da-Wei Niu}
\address{Department of Science,  Henan University of Animal Husbandry and Economy,
No. 6 North Longzihu Road, Henan 450046, Zhengdong District, Zhengzhou, PR China}
\email{nnddww@163.com}

\date{} 

\begin{document}

\begin{abstract}
In this paper we investigate the monotonicity properties related to the ratio of gamma functions, from which some related asymptotics and inequalities are established.  Some special cases also confirm the conjectures of C.-P.~Chen [Monotonicity properties, inequalities, and asymptotic expansions associated with the gamma function. Appl. Math. Comput. 283 (2016), 385--396.].
\end{abstract}

\maketitle

\section{Introduction}
To find an algebraic expression that takes the value $n!$ at each non-negative
integer $n$, Euler \cite{Euler} obtained the following well-defined convergent infinite product expression:
\begin{equation}\label{eq11}
\Gamma(1+x):=\prod_{k=1}^{\infty}\frac{k^{1-x}(k+1)^x}{k+x}=\prod_{k=1}^{\infty}\left(1+\frac{x}{k}\right)^{-1}\left(1+\frac{1}{k}\right)^{-x}, ~(x\ge 0).
\end{equation}
This expression $\Gamma(1+n)=n!$ holds for each non-negative
integer $n$ and the functional equation
\begin{equation}\label{eq12}
\Gamma(x)=x^{-1}\Gamma(x+1)
\end{equation}
for each $x\ge 1$. Since $\Gamma(1+x)$ is well defined for all $x\ge 0$, one can define a function $\Gamma(x)=x^{-1}\Gamma(1+x)$ for all $x>0$, and it is so called Euler's gamma function. Euler \cite{Euler} further found that $\Gamma(3/2)^2$ is the area of a circle with diameter $1$, which related to a work of Wallis \cite{Wallis}
\begin{equation}\label{eq13}
\Gamma^2\left(\frac{3}{2}\right)=\frac{2\cdot 4}{3\cdot 3}\cdot\frac{4\cdot 6}{5\cdot 5}\cdot\frac{6\cdot 8}{7\cdot 7}\dots =\frac{\pi}{4}.
\end{equation}

In view of \eqref{eq11} and \eqref{eq12}, formula \eqref{eq13} can be rewritten as
$$\sqrt{\pi}=\frac{1\cdot 2\cdot 3\dots n}{(1-1/2)\cdot(2-1/2)\dots (n-1/2)}\frac{\Gamma(n+1/2)}{\Gamma(n+1)},$$
for each integer $n\ge 1$. Therefore to approximate $\pi$ numerically, some problems are arising such as the asymptotics, inequalities and monotonicity of the ratio of two gamma functions. Those problems have a long history and have been investigated by many authors.  We refer the reader to, for examples, Tricomi and Erd\'{e}lyi \cite{MR0043948}, Fields \cite{MR0200486}, Frenzen \cite{MR883576, MR1147874}, Buri\'{c} and Elezovi\'{c} \cite{MR2784046} for asymptotic expansions, Bustoz and Ismail \cite{MR856710}, Alzer \cite{MR1388887}, and Chen and Qi \cite{MR2093060} for inequalities,  Qi and Chen \cite{MR2075188},  Koumandos \cite{MR2266574, MR2429884}, Koumandos and Pedersen \cite{MR3002606}  Mortici et al. \cite{MR3213683}, Chen and Paris \cite{MR3648484, MR3285558} and Chen \cite{MR3338412} for monotonicity. For more related works we refer the reader to the exhaustive survey \cite{MR2611044} of Qi and references therein.

Let $I\subset \rb$ be an open interval. We recall that an infinitely differentiable function $f(x)$ is said to be completely monotonic over $I$ if $(-1)^nf(x)\ge 0$ for all $x\in I$ and all integers $n\ge 0$. We also recall (Bernstein's Theorem) that $f$ is completely monotonic over $I\subset \rb$  if and only if
$$f(x)=\int_{I}e^{-xt}\,d\mu(t),$$
where $\mu(t)$ is a nonnegative measure such that the integral converges for all $x\in I$.

In \cite{MR2075188}, Qi and Chen proved that
$$\log\left[\frac{\Gamma(x+1)}{(x+1/4)\Gamma(x+1/2)}\right]$$
was completely monotonic over $(0,\infty)$. As a consequence, \cite{MR2075188} gave the following sharp inequality
$$\frac{1}{\sqrt{\pi n+4-\pi}}\le \frac{(1-1/2)\cdot(2-1/2)\dots (n-1/2)}{1\cdot 2\cdot 3\dots n}< \frac{1}{\sqrt{\pi n+\pi/4}},$$
for all integer $n\ge 1$. Motivated by \cite{MR2075188}, Mortici et al. \cite[Theorem 1, 2]{MR3213683} proved that the functions
$$\log\left[\frac{\left(\frac{1}{2\pi}\sqrt{3}\Gamma(2/3)\right)^3}{x^2\left(\frac{\Gamma(x+1/3)}{\Gamma(x+1)\Gamma(1/3)}\right)^3}\right]\;\mbox{and}\; \log\left[\frac{\left(\Gamma(2/3)\right)^{-3}}{x\left(\frac{\Gamma(x+2/3)}{\Gamma(x+1)\Gamma(2/3)}\right)^3}\right]$$
were completely monotonic on $(0,\infty),$ from which \cite[Corollary 1, 2]{MR3213683} gave the following inequalities
$$\frac{1/3}{\sqrt[3]{n^{2}}}\le \frac{(1-2/3)\cdot(2-2/3)\dots (n-2/3)}{1\cdot 2\cdot 3\dots n}< \frac{\frac{1}{2\pi}\sqrt{3}\Gamma(2/3)}{\sqrt[3]{n^{2}}}$$
and
$$\frac{2/3}{\sqrt[3]{n}}\le \frac{(1-1/3)\cdot(2-1/3)\dots (n-1/3)}{1\cdot 2\cdot 3\dots n}< \frac{1/\Gamma(2/3)}{\sqrt[3]{n}}.$$
Mortici et al. \cite[Theorem 3, 4]{MR3213683} also obtained further improvements of the above inequalities.

For each integer $n\ge 0$ and all $x\in\rb$, The $n$th Bernoulli polynomial $B_n(x)$ is defined by
\begin{equation}\label{def: B_n}
\sum_{n\ge 0}\frac{B_n(x)}{n!}t^n:=\frac{te^{tx}}{e^t-1}, ~~|t|<2\pi,
\end{equation}
and the the $n$th Bernoulli number $B_n:=B_n(0)$. Chen and Paris \cite{MR3648484} considered
\begin{equation*}
F_m(t)=(-1)^m\left(\log\left(\frac{\Gamma(t+1)}{t^{1/2}\Gamma(t+1/2)}\right)-\sum_{j=1}^m\left(1-\frac{1}{2^{2\ell}}\right)\frac{B_{2\ell}}{\ell(2\ell-1)}\frac{1}{t^{2\ell-1}}\right)
\end{equation*}
and proved the completely monotonicity of $F_m(t)$ on $(0,\infty)$ for each integer $m\ge 0$. Chen in \cite{MR3338412} further considered for all integer $m\ge 0$ the functions
\begin{equation*}
U_m(t):=(-1)^m\left(\log\left(\frac{\Gamma(t+2/3)}{t^{1/3}\Gamma(t+1/3)}\right)-\sum_{j=1}^m\frac{B_{2\ell+1}(1/3)}{\ell(2\ell+1)}\frac{1}{t^{2\ell}}\right)
\end{equation*}
and
\begin{equation*}
V_m(t):=(-1)^m\left(\log\left(\frac{\Gamma(t+3/4)}{t^{1/2}\Gamma(t+1/4)}\right)-\sum_{j=1}^m\frac{B_{2\ell+1}(1/4)}{\ell(2\ell+1)}\frac{1}{t^{2\ell}}\right).
\end{equation*}
Here and throughout this paper, the empty sum, such as $\sum_{1\le j\le 0}:=0$. He then conjectured that $U_m(t)$ and $V_m(t)$ are completely monotonic on $(0,\infty)$ and verified the cases of $m\in\{0,1,2,3\}$. The conjecture for $V_m(t)$ has been proved by Chen and Paris in \cite{MR3285558}.

Motivated by Chen and Paris \cite{MR3648484} and Chen \cite{MR3338412}, we consider the following functions:
\begin{equation}\label{def:G}
G_m(x,t):=(-1)^{m-1}\left(\log\left(\frac{\Gamma(t+1-x)\Gamma(t+x)}{\Gamma(t)\Gamma(t+1)}\right)-\sum_{\ell=1}^m\frac{B_{2\ell}(x)-B_{2\ell}}{\ell(2\ell-1)}\frac{1}{t^{2\ell-1}}\right)
\end{equation}
and
\begin{equation}\label{def:H}
H_m(x,t):=(-1)^m\left(\log\left(\frac{\Gamma(t+1-x)}{t^{1-2x}\Gamma(t+x)}\right)-\sum_{\ell=1}^m\frac{B_{2\ell+1}(x)}{\ell(2\ell+1)}\frac{1}{t^{2\ell}}\right),
\end{equation}
for each integer $m\ge -1$, $x\in[0,1/2]$ and $t>0$. It is easy to check that $G_m(1/2,t)=2F_m(t)$, $H_m(1/3,t)=U_m(t)$ and $H_m(1/4,t)=V_m(t)$.

In this paper, we prove some inequalities related to $G_m(x,t)$ and $H_m(x,t)$. For example (Theorem \ref{main0}), we prove
\begin{align*}
\exp&\left(\sum_{\ell=1}^{2m_1+1}\frac{B_{2\ell}(x)-B_{2\ell}}{\ell(2\ell-1)}\frac{1}{t^{2\ell-1}}\right)\\
&\qquad\qquad \qquad\qquad < \frac{\Gamma(t+1-x)\Gamma(t+x)}{\Gamma(t)\Gamma(t+1)}<  \exp\left(\sum_{\ell=1}^{2m_2}\frac{B_{2\ell}(x)-B_{2\ell}}{\ell(2\ell-1)}\frac{1}{t^{2\ell-1}}\right)
\end{align*}
and
$$\exp\left(\sum_{\ell=1}^{2m_1}\frac{B_{2\ell+1}(x)}{\ell(2\ell+1)}\frac{1}{t^{2\ell}}\right) < \frac{\Gamma(t+1-x)}{t^{1-2x}\Gamma(t+x)}< \exp\left(\sum_{\ell=1}^{2m_2+1}\frac{B_{2\ell+1}(x)}{\ell(2\ell+1)}\frac{1}{t^{2\ell}}\right),$$
hold for any given integers $m_1,m_2\ge 0$, $x\in(0,1/2)$ and $t>0$. We further prove that $G_m(x,t)$ and $H_m(x,t)$ are completely monotonic on $t\in (0,\infty)$ for each $x\in[0,1/2]$ and each integer $m\ge 0$. Some special cases confirm the conjecture of Chen \cite{MR3338412}.

We point out that the proof of the main results of this paper is self-contained. It just requires the infinite product expression \eqref{eq11} of the gamma function, and some basic knowledge of Fourier series and calculus.


\section{Primary results}
We first use basic knowledge of Fourier analysis to prove
\begin{proposition}\label{pro23}For all $u>0$ we have
\begin{equation}\label{eq26}
\frac{u\cosh[(x-1/2)u]}{2\sinh[u/2]}=1+2\sum_{m\ge 1}\frac{u^2\cos(2\pi mx)}{u^2+4\pi^2m^2},~~x\in[0,1].
\end{equation}
\end{proposition}
\begin{proof}We first note that for all $u>0$ and $x\in[0,1]$,
\begin{align*}
\frac{\cosh[(x-1/2)u]}{\sinh[u/2]}&=\frac{e^{x u}+e^{(1-x)u}}{e^u-1}\\
&=\sum_{n\ge 0}e^{-(n+1-x)u}+\sum_{n\ge 0}e^{-(n+x)u}=\sum_{n\in\zb}e^{-|n+x|u}.
\end{align*}
Now, let $u>0$ and $x\in\rb$, we define
$$f(x)=\sum_{n\in\zb}e^{-|n+x|u}.$$
Clearly, $f(x)$ is a periodic function with period $1$ about $x\in\rb$, and continuous in $\rb$. Hence by Fourier analysis, for all $x\in\rb$,
\begin{align*}
f(x)&=\sum_{m\in\zb}e^{2\pi\ri mx}\int_{0}^{1}f(v)e^{-2\pi\ri mv}\,dv\\
&=\sum_{m\in\zb}e^{2\pi\ri mx}\int_{\rb}e^{-|v|u-2\pi\ri mv}\,dv=\sum_{m\in\zb}e^{2\pi\ri mx}\left(\frac{1}{u+2\pi\ri m}+\frac{1}{u-2\pi\ri m}\right)\\
&=2u\sum_{m\in\zb}\frac{e^{2\pi\ri mx}}{u^2+4\pi^2m^2}=\frac{2}{u}+4u\sum_{m\ge 1}\frac{\cos(2\pi mx)}{u^2+4\pi^2m^2}.
\end{align*}
This completes the proof.
\end{proof}
We need the following Proposition.
\begin{proposition}\label{pro22}For $-2\pi<u<2\pi$ and $x$ be real number we have
\begin{equation}\label{eq23}
\frac{u\cosh[(x-1/2)u]}{2\sinh[u/2]}=\sum_{k\ge 0}\frac{B_{2k}(x)}{(2k)!}u^{2k}
\end{equation}
and
\begin{equation}\label{eq24}
\frac{\sinh[(x-1/2)u]}{2\sinh[u/2]}=\sum_{k\ge 0}\frac{B_{2k+1}(x)}{(2k+1)!}u^{2k}.
\end{equation}
In particular, for all integers $k\ge 0$ we have,
\begin{equation}\label{eq25}
(k+1)B_k(x)=B_{k+1}'(x).
\end{equation}
\end{proposition}
\begin{proof}Using definition \ref{def: B_n}, we have
\begin{equation}\label{eq261}
\frac{u\cosh[(x-1/2)u]}{\sinh[u/2]}=\frac{u(e^{x u}+e^{(1-x)u})}{e^u-1}=\sum_{k\ge 0}\frac{B_k(x)+B_k(1-x)}{k!}u^{k},
\end{equation}
and
\begin{equation}\label{eq27}
\frac{\sinh[(x-1/2)u]}{\sinh[u/2]}=\frac{e^{x u}-e^{(1-x)u}}{e^u-1}=\sum_{k\ge 0}\frac{B_k(x)-B_k(1-x)}{k!}u^{k-1}.
\end{equation}
Since the functions in \eqref{eq261} and \eqref{eq27} are even in $u$, we have
\begin{equation}\label{eq281}
B_k(x)=(-1)^{k}B_k(1-x)
\end{equation}
for all integers $k\ge 0$, and the proofs of \eqref{eq23} and \eqref{eq24} follows. Finally, by use of the fact that
$(\cosh(t))'=\sinh(t)$, $(\sinh(t))'=\cosh(t)$, it is easy to prove \eqref{eq25} by \eqref{eq23} and \eqref{eq24}. This completes the proof.
\end{proof}
We now deduce the following Corollary by Proposition \ref{pro23} and Proposition \ref{pro22}.
\begin{corollary}\label{cor21} We have for all $\ell\in\zb_{\ge 0}$,
$$\sum_{m\ge 1}\frac{\cos(2\pi mx)}{m^{2\ell+2}}=\frac{(-1)^{\ell}B_{2\ell+2}(x)(2\pi)^{2\ell+2}}{2(2\ell+2)!},~~x\in[0,1];$$
and
$$\sum_{m\ge 1}\frac{\sin(2\pi mx)}{m^{2\ell+1}}=\frac{(-1)^{\ell-1}B_{2\ell+1}(x)(2\pi)^{2\ell+1}}{2(2\ell+1)!},~~x\in(0,1).$$
\end{corollary}
\begin{proof}In Corollary \ref{cor21}, it is easy to check that the series
$\sum_{m\ge 1}\frac{\sin(2\pi mx)}{m^{2\ell+1}}$
converges uniformly on $x\in[\varepsilon, 1-\varepsilon]$ with any given $\varepsilon\in(0,1/2)$, hence the second formula is the derivative of the first formula with respect to $x$, so we just need prove the first formula. By using \eqref{eq26} of Proposition \ref{pro23}, we have
\begin{align*}
\frac{u\cosh[(x-1/2)u]}{2\sinh[u/2]}&=1+2\sum_{m\ge 1}\frac{u^2}{u^2+(2\pi m)^2}\cos(2\pi mx)\\
&=1+2\sum_{m\ge 1}\cos(2\pi mx)\sum_{\ell\ge 0}\frac{(-1)^{\ell}u^{2\ell+2}}{(2\pi m)^{2\ell+2}}\\
&=1+2\sum_{\ell\ge 0}u^{2\ell+2}\frac{(-1)^{\ell}}{(2\pi)^{2\ell+2}}\sum_{m\ge 1}\frac{\cos(2\pi mx)}{m^{2\ell+2}}.
\end{align*}
Then the proof follows from \eqref{eq23} of Proposition \ref{pro22}.
\end{proof}

We now state the first theorem of this paper.
\begin{theorem}\label{main}Define for all integers $L\ge -1$, $x\in[0,1]$ and $u>0$ that
\begin{equation*}
S_L(x,u):=(-1)^{L}\left(\frac{\sinh[(x-1/2)u]}{2\sinh[u/2]}-\sum_{0\le k\le L}\frac{B_{2k+1}(x)}{(2k+1)!}u^{2k}\right).
\end{equation*}
Then, we have
$$S_{-1}(0,u)=1/2,\; S_{-1}(1/2,u)=0,\; S_{-1}(x,u)>0,$$
and
$$S_{L}(0,u)=S_{L}(1/2,u)=0,\; S_L(x,u)> 0$$
hold for all integers $L\ge 0$ and all $x\in(0,1/2)$.
\end{theorem}
A special case of Theorem \ref{main} confirms a conjecture of Chen \cite[Conjecture 1]{MR3338412}. In fact, let $\mu_m(t)$ and $\nu_m(t)$ be defined as in Chen \cite[Conjecture 1]{MR3338412}. Clearly, $\mu_m(t)=2S_{m}(1/3,t)$ and $\nu_m(t)=2S_{m}(1/4,t)$, and the proof follows.

To prove Theorem \ref{main}, we first prove
\begin{lemma}\label{klem}For all integers $L\ge -1$, $x\in(0,1)$ and $u>0$, we have
$$S_L(x,2\pi u)=\frac{u^{2L+2}}{\pi}\sum_{m\ge 1}\frac{\sin(2\pi mx)}{m^{2L+1}(m^2+u^2)}.$$
\end{lemma}
\begin{proof}Take the derivative in \eqref{eq26} of Proposition \ref{pro23} with  respect to $x$,
\begin{equation*}
\frac{\sinh[(x-1/2)u]}{2\sinh[u/2]}=-2\sum_{m\ge 1}\frac{(2\pi m)\sin(2\pi mx)}{u^2+(2\pi m)^2},~~x\in(0,1).
\end{equation*}
By Corollary \ref{cor21}, we have
\begin{align*}
\sum_{k=0}^{L}\frac{B_{2k+1}(x)}{(2k+1)!}u^{2k}&=\sum_{k=0}^{L}\frac{(-1)^{k-1}}{\pi}\left(\frac{u}{2\pi}\right)^{2k}\sum_{m\ge 1}\frac{\sin(2\pi mx)}{m^{2k+1}}\\
&=-\sum_{m\ge 1}\frac{\sin(2\pi mx)}{\pi m}\frac{1-(-(\frac{u}{2\pi m})^2)^{L+1}}{1+(\frac{u}{2\pi m})^2}.
\end{align*}
Therefore,
$$
S_L(x,u)=\sum_{m\ge 1}\frac{\sin(2\pi mx)}{\pi m}\frac{(\frac{u}{2\pi m})^{2L+2}}{1+(\frac{u}{2\pi m})^2},~~x\in(0,1),
$$
that is,
\begin{equation*}
S_L(x,2\pi u)=\frac{u^{2L+2}}{\pi}\sum_{m\ge 1}\frac{\sin(2\pi mx)}{m^{2L+1}(m^2+u^2)},~~x\in(0,1).
\end{equation*}
This completes the proof of the lemma.
\end{proof}
We now prove Theorem \ref{main}. Let $u>0$ be any given number. By using the definition of $S_L(x,u)$ and the fact of $B_{2k+1}(1/2)=-B_{2k+1}(1-1/2)=0$ for all integers $k\ge 0$ (see~ \eqref{eq281}), we have
$$S_L(1/2,u):=(-1)^{L}\left(\frac{\sinh[(1/2-1/2)u]}{2\sinh[u/2]}-\sum_{0\le k\le L}\frac{B_{2k+1}(1/2)}{(2k+1)!}u^{2k}\right)=0.$$
Also by using \eqref{eq27}, we have $B_1(0)=-1/2$ and $B_{2k+1}(0)=0, k\in\zb_{\ge 1}$. Thus using the definition of $S_L(x,u)$ implies
$$S_L(0,u):=(-1)^{L}\left(\frac{\sinh[(0-1/2)u]}{2\sinh[u/2]}-\sum_{0\le k\le L}\frac{B_{2k+1}(0)}{(2k+1)!}u^{2k}\right),$$
that is
$$
S_{L}(0,u)=\begin{cases} 1/2,\quad &L=-1,\\
~~\; 0 &L\in\zb_{\ge 1}.
\end{cases}
$$
For the cases of $x\in(0,1/2)$ we shall using mathematical induction on $L$. The result for $L=-1$ is obvious. By Lemma \ref{klem} we have
$$\frac{\,d^2}{\,d x^2}S_{L+1}(x,u)=-(2\pi u)^2S_{L}(x,u)$$
for all $x\in(0,1/2)$.  Thus if $S_{L}(x,u)>0$ holds for all $x\in(0,1/2)$, then $\frac{\,d^2}{\,d x^2}S_{L+1}(x,u)<0$ and hence $\frac{\,d}{\,d x}S_{L+1}(x,u)$ is a decreasing function on $x\in(0,1/2)$. Moreover, since
$$\frac{\,d}{\,d x}S_{L+1}(x,u)=\frac{u^{2L+4}}{\pi}\sum_{m\ge 1}\frac{2\pi\cos(2\pi mx)}{m^{2L+3}(m^2+u^2)},$$
it is easy to find that
$$\frac{\,d}{\,d x}\big|_{x=0}S_{L+1}(x,u)=2u^{2L+4}\sum_{m\ge 1}\frac{1}{m^{2L+3}(m^2+u^2)}>0,$$
and
$$\frac{\,d}{\,d x}\big|_{x=1/2}S_{L+1}(x,u)=2u^{2L+4}\sum_{m\ge 1}\frac{(-1)^m}{m^{2L+3}(m^2+u^2)}<0.$$
So there exists only one point $x_m(L)\in(0,1/2)$ such that $\frac{\,d}{\,d x}\big|_{x=x_m(L)}S_{L+1}(x,u)=0$, and which also is the maximum value point of $S_{L+1}(x,u)$ for $x\in(0,1/2)$. Further more, $S_{L+1}(x,u)$ is increasing on $(0,x_m(L))$ and decreasing on $[x_m(L),1/2)$. Finally, by noting that $S_{L+1}(0,u)=S_{L+1}(1/2,u)=0$ holds for all integers $L\ge 0$, we finish the proof of the theorem.

\section{Main results}

\begin{proposition}\label{propm} Let $m\ge 0$ be any given integer, $x\in[0,1/2]$ and $t>0$. We have
\begin{equation}\label{eq21}
H_m(x,t)=2\int_{0}^{\infty}\frac{S_{m}(x,u)}{u}e^{-ut}\,du
\end{equation}
and
\begin{equation}\label{eq22}
G_{m}(x,t)=2\int_{0}^x\int_{0}^{\infty}S_{m-1}(y,u)e^{-ut}\,du\,dy+(-1)^{m}\delta_{m0}x(1-x)/t.
\end{equation}
Here $\delta_{mn}$ is the usual Kronecker delta.
\end{proposition}
\begin{proof}
Let $H_m^*(x,t)$ denoting the right hand side of \eqref{eq21}. By the fact that $H_m^*(0,t)=H_m(0,t)=0$, we just need to prove
\begin{equation}\label{eq20}
\frac{\P}{\P x}H_m^*(x,t)=\frac{\P}{\P x}H_m(x,t),
\end{equation}
then the proof of \eqref{eq21} follows. Since
$$\log y=\int_{1}^y\frac{1}{v}\,dv=\int_{1}^y\,dv\int_{0}^{\infty}e^{-vu}\,du=\int_{0}^{\infty}\frac{e^{-u}-e^{-uy}}{u}\,du, ~y>0,$$
we have
\begin{align}\label{eqdlg}
-\frac{\,d}{\,dz}\log\Gamma(z)&=\sum_{n\ge 0}\left(\frac{1}{n+z}+\log\left(n+1\right)-\log\left(n+2\right)\right)\nonumber\\
&=\sum_{n\ge 0}\left(\int_{0}^{\infty}e^{-(n+z)u}\,du-\int_{0}^{\infty}\frac{e^{-(n+1)u}-e^{-(n+2)u}}{u}\,du\right)\nonumber\\
&=\int_{0}^{\infty}\left(\frac{e^{-uz}}{1-e^{-u}}-\frac{1-e^{-u}}{u}\frac{e^{-u}}{1-e^{-u}}\right)\,du\nonumber\\
&=\int_{0}^{\infty}\left(\frac{e^{-uz}}{1-e^{-u}}-\frac{e^{-u}}{u}\right)\,du,
\end{align}
for all $z>0$, by using the definition \eqref{eq11}. Combining \eqref{eqdlg} with \eqref{def:H} and \eqref{eq25} we have
\begin{align*}
(-1)^m\frac{\P}{\P x}H_m(x,t)=&\frac{\,d}{\,dx}\left(\log\left(\frac{\Gamma(t+1-x)}{t^{1-2x}\Gamma(t+x)}\right)-\sum_{\ell=1}^m\frac{B_{2\ell+1}(x)}{\ell(2\ell+1)}\frac{1}{t^{2\ell}}\right)\\
=&\frac{\,d}{\,dx}\log \Gamma(t+1-x)-\frac{\,d}{\,dx}\log \Gamma(t+x)+2\log t-\sum_{\ell=1}^m\frac{B_{2\ell}(x)}{\ell}\frac{1}{t^{2\ell}}\\
=&\int_{0}^{\infty}\left(\frac{e^{-(t+1-x)u}}{1-e^{-u}}+\frac{e^{-(t+x)u}}{1-e^{-u}}-\frac{2e^{-u}}{u}\right)\,du\\
&+2\int_{0}^{\infty}\frac{e^{-u}-e^{-ut}}{u}\,du
-2\sum_{\ell=1}^m\frac{B_{2\ell}(x)}{(2\ell)!}\int_{0}^{\infty}e^{-ut}u^{2\ell -1}\,du,
\end{align*}
that is,
\begin{align*}
(-1)^m\frac{\P}{\P x}H_m(x,t)
=&\int_{0}^{\infty}\left(\frac{e^{x u}+e^{(1-x)u}}{e^u-1}-\frac{2}{u}-2\sum_{\ell=1}^m\frac{B_{2\ell}(x)}{(2\ell)!}u^{2k-1}\right)e^{-ut}\,du\\
=&2\int_{0}^{\infty}\left(\frac{u\cosh[(x-1/2)u]}{2\sinh[u/2]}-\sum_{\ell=0}^m\frac{B_{2\ell}(x)}{(2\ell)!}u^{2k}\right)\frac{e^{-ut}}{u}\,du\\
=&\int_{0}^{\infty}(-1)^m\frac{\P}{\P x}\frac{S_{m}(x,u)}{u}e^{-ut}\,du=(-1)^m\frac{\P}{\P x}H_m^*(x,t),
\end{align*}
by using the definition of $S_{m}(x,u)$, we complete the proof of \eqref{eq20} and the proof of \eqref{eq21} follows. The proof of \eqref{eq22} is directly obtained by using \eqref{eq21}, the definitions \eqref{def:G} and \eqref{def:H}. In fact, by noting $B_1(y)=y-1/2$ and $2\ell B_{2\ell-1}(y)=B_{2\ell}'(y)$, we have
\begin{align*}
G_{m}(x,t)=&(-1)^{m-1}\int_{0}^x\frac{\,d}{\,dy}\left(\log\left(\Gamma(t+1-y)\Gamma(t+y)\right)-\sum_{\ell=1}^{m}\frac{B_{2\ell}(y)}{\ell(2\ell-1)}\frac{1}{t^{2\ell-1}}\right)\,dy\\
=&(-1)^{m}\int_{0}^x\frac{\,d}{\,dt}\bigg(\log\left(\frac{\Gamma(t+1-y)}{t^{1-2y}\Gamma(t+y)}\right)\\
&\qquad\qquad\qquad\qquad\qquad\qquad\quad-\sum_{\ell=1}^{m-1}\frac{B_{2\ell+1}(y)}{\ell(2\ell+1)}\frac{1}{t^{2\ell}}-\delta_{m0}B_{1}(y)\log t^2\bigg)\,dy\\
=&-\int_{0}^x\frac{\,d}{\,dt}H_{m-1}(y,t)\,dy-(-1)^m\frac{2\delta_{m0}}{t}\int_{0}^xB_1(y)\,dy\\
=&2\int_{0}^x\int_{0}^{\infty}S_{m-1}(y,u)e^{-ut}\,du\,dy+(-1)^{m}\frac{\delta_{m0}}{t}x(1-x).
\end{align*}
This completes the proof.
\end{proof}
As a corollary, we prove
\begin{corollary}\label{cor1} We have
$$H_0(0,t)=H_0(1/2,t)=0,\; \mbox{and}\; G_0(0,t)>0, G_0(1/2,t)>0;$$
and
$$H_m(0,t)=H_m(1/2,t)=G_m(0,t)=G_m(1/2,t)=0$$
hold for all integers $m\ge 1$; and
$$H_m(x,t)>0, G_m(x,t)>0$$
hold for all $x\in(0,1/2)$ and all integers $m\ge 0$.
\end{corollary}
\begin{proof}
The proof of the corollary is a direct use of Theorem \ref{main} and Proposition \ref{propm}, and we shall omit it.
\end{proof}

We now state and prove the main results of this paper.
We have the following inequality for ratios of gamma functions.
\begin{theorem}\label{main0}Let $m_1,m_2\ge 0$ be any given integers, $t>0$ and $x\in(0,1/2)$. We have
\begin{align*}
\exp&\left(\sum_{\ell=1}^{2m_1+1}\frac{B_{2\ell}(x)-B_{2\ell}}{\ell(2\ell-1)}\frac{1}{t^{2\ell-1}}\right)\\
&\qquad\qquad \qquad\qquad < \frac{\Gamma(t+1-x)\Gamma(t+x)}{\Gamma(t)\Gamma(t+1)}<  \exp\left(\sum_{\ell=1}^{2m_2}\frac{B_{2\ell}(x)-B_{2\ell}}{\ell(2\ell-1)}\frac{1}{t^{2\ell-1}}\right)
\end{align*}
and
$$\exp\left(\sum_{\ell=1}^{2m_1}\frac{B_{2\ell+1}(x)}{\ell(2\ell+1)}\frac{1}{t^{2\ell}}\right) < \frac{\Gamma(t+1-x)}{t^{1-2x}\Gamma(t+x)}< \exp\left(\sum_{\ell=1}^{2m_2+1}\frac{B_{2\ell+1}(x)}{\ell(2\ell+1)}\frac{1}{t^{2\ell}}\right).$$
\end{theorem}
\begin{proof}Using the definitions \eqref{def:G} and \eqref{def:H}, we have
$$\frac{\Gamma(t+1-x)\Gamma(t+x)}{\Gamma(t)\Gamma(t+1)}=\exp\left(\sum_{\ell=1}^{m}\frac{B_{2\ell}(x)-B_{2\ell}}{\ell(2\ell-1)}\frac{1}{t^{2\ell-1}}-(-1)^mG_m(x,t)\right),$$
and
$$\frac{\Gamma(t+1-x)}{t^{1-2x}\Gamma(t+x)}=\exp\left(\sum_{\ell=1}^{m}\frac{B_{2\ell+1}(x)}{\ell(2\ell+1)}\frac{1}{t^{2\ell}}+(-1)^mH_m(x,t)\right)$$
hold for each integer $m\ge 0$. Therefore the proof follows from Corollary \ref{cor1}.
\end{proof}

\begin{theorem}\label{main1}
For each $m\in\zb_{\ge 0}$ and each $x\in[0,1/2]$, the functions $G_m(x,t)$ and $H_m(x,t)$ are completely monotonic on $t\in(0,\infty)$.
\end{theorem}
\begin{proof}Let $x\in[0,1/2]$ and $t>0$. Using Theorem \ref{main} and Proposition \ref{propm}, we have
\begin{align*}
(-1)^k\frac{\,d^k }{\,d t^k}H_m(x,t)&=2(-1)^k\int_{0}^{\infty}\frac{\,d^k }{\,d t^k}\left(\frac{S_{m}(x,u)}{u}e^{-ut}\right)\,du\\
&=2\int_{0}^{\infty}S_{m}(x,u)u^{k-1}e^{-ut}\,du\ge 0,
\end{align*}
and
\begin{align*}
(-1)^k\frac{\,d^k }{\,d t^k}G_m(x,t)=&2(-1)^k\int_{0}^x\,dy\int_{0}^{\infty}\frac{\,d^k }{\,d t^k}\left(S_{m-1}(y,u)e^{-ut}\right)\,du\\
&+(-1)^{m}\frac{k!\delta_{m0}}{t^{k+1}}x(1-x)\\
=&2\int_{0}^x\,dy\int_{0}^{\infty}S_{m-1}(y,u)e^{-ut}u^{k}\,du+(-1)^{m}\frac{k!\delta_{m0}}{t^{k+1}}x(1-x)\ge 0,
\end{align*}
hold for all integers $k,m\ge 0$. This completes the proof.

\end{proof}

\end{document}